\documentclass[12pt]{amsart}
\usepackage[margin=1in]{geometry}  
\usepackage{graphicx}              
\usepackage{amsmath}               
\usepackage{amsfonts}              
\usepackage{amsthm}                
\usepackage{caption} 
\usepackage{subcaption}
\usepackage{tikz}
\usepackage{tabu}
\usepackage[shortlabels]{enumitem}

\newtheorem{thm}{Theorem}[section]
\newtheorem{lem}[thm]{Lemma}

\newtheorem{prop}[thm]{Proposition}
\newtheorem{cor}[thm]{Corollary}

\newtheorem{example}[thm]{Example}

\theoremstyle{definition}
\newtheorem{definition}[thm]{Definition}
\newtheorem{remark}[thm]{Remark}
\newtheorem{question}[thm]{Question}

\newcommand{\ZZ}{\mathbb{Z}}      

\newcommand{\QQ}{\mathbb{Q}}

\begin{document}

\title{Using rational homology circles to construct rational homology balls}
\author{Jonathan Simone}

\begin{abstract}
Motivated by Akbulut-Larson's construction of Brieskorn spheres bounding rational homology 4-balls, we explore plumbed 3-manifolds that bound rational homology circles and use them to construct infinite families of rational homology 3-spheres that bound rational homology 4-balls. Some of these rational homology 3-spheres are new examples of integer homology 3-spheres that bound rational homology 4-balls, but do not bound integer homology 4-balls (i.e. nontrivial elements of $\text{ker}(\Theta_{\ZZ}^3\to\Theta_{\QQ}^3)$). In particular, we find infinite families of torus bundles over the circle that bound rational homology circles, provide a simple method for constructing more general plumbed 3-manifolds that bound rational homology circles, and show that, for example, $-1$-surgery along any unknotting number one knot $K$ with a positive crossing that can be switched to unknot $K$ bounds a rational homology 4-ball.
\end{abstract}

\maketitle 

\section{Introduction}\label{intro}

Understanding which rational homology 3-spheres ($\QQ S^3s$) bound rational homology 4-balls ($\QQ B^4s$) is a widely explored open question among Kirby's list of problems (Problem 4.5 in \cite{kirbyproblemlist}). Certain classifications of $\QQ S^3s$ bounding $\QQ B^4s$ do exist (e.g., lens spaces \cite{liscalensspace}, certain small Seifert fibered spaces \cite{lecuonamontesinosknots}, some Dehn surgeries on knots \cite{acetogolla}, and some Brieskorn spheres \cite{akbulutlarson}, \cite{fickle}, \cite{cassonharer}), but the question at large is far from resolved. In \cite{akbulutlarson}, Akbulut-Larson used the fact that $0$-surgery on the figure-eight knot bounds a rational homology circle ($\QQ S^1\times B^3$) (\cite{fintushelsternmuinvt}) to construct infinite families of Brieskorn spheres that bound $\QQ B^4s$, and in \cite{savkballs}, \c{S}avk produced even more examples. Their constructions rely in part on the following lemma, which Akbulut-Larson proved for the case of $0$-surgery on the figure-eight knot (and, more generally, rationally slice knots). The lemma presented below is a more general version.


\begin{lem} Let $Y$ bound a $\mathbb{Q}S^1\times B^3$, $W$, and let $K$ be a knot in $Y$ such that $[K]$ has infinite order in $H_1(Y;\mathbb{Z})$. Then the 4-manifold obtained by attaching a 2-handle to $W$ along $K$ is a $\mathbb{Q}B^4$. Consequently, any integer surgery along $K$ yields a $\QQ S^3$ that bounds a $\QQ B^4$. \label{rationalballobstruction}\end{lem}


The upshot of this construction is that a single $\QQ S^1\times B^3$ can be used to construct infinite families of $\QQ S^3s$ bounding $\QQ B^4s$, as was the case in the Akbulut-Larson paper. For example, using the notation of Lemma \ref{rationalballobstruction}, we can add a 2-handle to $W$ along $K$ with any integer framing, yielding an infinite family of $\QQ S^3s$ bounding $\QQ B^4s$. Moreover, $K$ can be any knot having infinite order in $H_1(Y;\ZZ)$; each such knot provide an infinite family of $\QQ S^3s$ bounding $\QQ B^4s$, as above. 

The more specific problem of finding integer homology spheres ($\ZZ S^3s$) that bound $\QQ B^4s$ but do not bound integer homology balls ($\ZZ B^4s$) is also of significant interest. The first such example, discovered by Fintushel-Stern in \cite{fintushelsternmuinvt}, was the Brieskorn sphere $\Sigma(2,3,7)$. 
This gave the first nontrivial element in the kernel of the natural inclusion homomorphism $\Theta_{\ZZ}^3\to\Theta_{\QQ}^3$ from the integer homology cobordism group of $\ZZ S^3s$ to the rational homology cobordism group of $\QQ S^3s$. In the recent papers of Akbulut-Larson \cite{akbulutlarson} and \c{S}avk \cite{savkballs} mentioned above, infinitely many new nontrivial elements in $\text{ker}(\Theta_{\ZZ}^3\to\Theta_{\QQ}^3)$ were constructed. Using the same $\QQ S^1\times B^3$ that they used (which was bounded by 0-surgery on the figure-eight knot), we construct an infinite family of integer surgeries along knots that bound $\QQ B^4s$ and show that infinitely many of them are nontrivial elements in $\text{ker}(\Theta_{\ZZ}^3\to\Theta_{\QQ}^3)$.

\begin{thm}
If $K$ is an unknotting number one knot with a positive (resp. negative) crossing that can be switched to unknot $K$, then $S^3_{-1}(K)$ (resp. $S^3_1(K)$) bounds a $\QQ B^4$. In particular, if $K_n$ is  the twist knot shown in Figure \ref{twist} with $n$ odd, then $S^3_{-1}(K_n)$ bounds a $\QQ B^4$, but does not bound a $\ZZ B^4$, giving infinitely many nontrivial elements in $\text{ker}(\Theta_{\ZZ}^3\to\Theta_{\QQ}^3)$.
\label{thm1}
\end{thm}

\begin{figure}
	\centering
	\includegraphics[scale=.65]{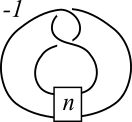}
	\caption{Let $K_n$ be the twist knot depicted, where the box labeled $n$ indicates $n$ full twists. If $n$ is odd, then $S^3_{-1}(K_n)$ bounds a $\QQ B^4$ but not an $\ZZ B^4$}
	\label{twist}
\end{figure}

\begin{remark} It follows from Theorem \ref{thm1} that $-1$-surgery (resp. $1$-surgery) along any twisted positively-clasped (resp. negatively-clasped) Whitehead double of any knot in $S^3$ bounds a $\QQ B^4$. \end{remark}

\begin{remark} All that is known about the structure of $\text{ker}(\Theta_{\ZZ}^3\to\Theta_{\QQ}^3)$ is that it contains a subgroup isomorphic to $\ZZ$. In particular, it is unknown whether some subset of  the examples of Akbulut-Larson, \c{S}avk, or Theorem \ref{thm1} give rise to linearly independent elements in $\text{ker}(\Theta_{\ZZ}^3\to\Theta_{\QQ}^3)$. An affirmative answer would show that the kernel is larger that $\ZZ$.\end{remark}

As mentioned above, a single $\QQ S^1\times B^3$ can be used to construct infinitely many $\QQ S^3s$ bounding $\QQ B^4s$. Hence finding simple 3-manifolds that bound $\QQ S^1\times B^3s$ is a worthwhile pursuit. Since the boundaries of $\QQ S^1\times B^3s$ have the rational homology of $S^1\times S^2$ ($\QQ S^1\times S^2$), it is natural to ask:

\begin{question} Which $\QQ S^1\times S^2s$ bound $\QQ S^1\times B^3s$? \end{question}

The first class of $\QQ S^1\times S^2s$ that one might consider is the set of $0$-surgeries along knots in $S^3$ (which are $\ZZ S^1\times S^2s$). Classifying which knots admit $0$-surgeries that bound $\QQ S^1\times B^3s$ is equivalent to classifying rationally slice knots (\cite{cochranetal}). There are many known families of rationally slice knots: strongly negative amphichiral knots (including the figure-eight knot) (\cite{kawauchinegamph}), Miyazaki knots (\cite{kimwurationalslice}), and $(p,1)$-cables of rationally slice knots for $p>0$ (\cite{cochranetal}). 


We will instead focus on plumbed 3-manifolds---the boundaries of plumbings of $D^2$-bundles over $S^2$---whose associated weighted graphs have a single cycle. Each edge of the cycle must be decorated with either $``+"$ or $``-"$ to specify the sign of the intersection of the (oriented) base spheres. By changing the orientations of the base spheres and fibers of select disk bundles, it can be arranged that either all edges of the cycle are decorated with ``+" or all but one edge are decorated with ``+." If the cycle can be decorated with only ``+" signs, we say the plumbing is \textit{positive}. Otherwise, we say it is \textit{negative}.

We first focus on \textit{cyclic} plumbings, whose associated graphs is a cycle, as in Figure \ref{cyclic}. The boundaries of such plumbings are $T^2$-bundles over $S^1$ (c.f. \cite{neumann}). After endowing $T^2\times[0,1]=\mathbb{R}^2/\mathbb{Z}^2\times[0,1]$ with the coordinates $(\textbf{x},t)=(x,y,t)$, any $T^2$-bundle over $S^1$ is of the form $T^2\times[0,1]/(\textbf{x},1)\sim(\pm A\textbf{x},0)$, where $A\in SL(2,\ZZ)$, which is well-defined up to conjugation. The matrix $A$ is called the \textit{monodromy} of the torus bundle. A torus bundle is called \textit{elliptic} if $|\text{tr}A|<2$, \textit{parabolic} if $|\text{tr}A|=2$, and \textit{hyperbolic} if $|\text{tr}A|>2$. Moreover, it is called $\textit{positive}$ if $\text{tr}A>0$ and $\textit{negative}$ if $\text{tr}A<0$. Throughout, we will express the monodromy in terms of the generators $T=\begin{bmatrix} 1&1\\0&1\end{bmatrix}$ and $S=\begin{bmatrix} 0&1\\-1&0\end{bmatrix}$. 

\begin{figure}
	\centering
	\begin{subfigure}{.45\textwidth}
		\centering
		\includegraphics[scale=.55]{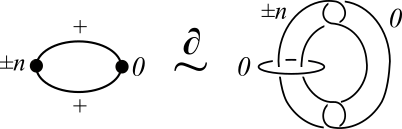}
		\caption{Cyclic plumbing with boundary $\textbf{T}_{\pm n}$}\label{par0}
	\end{subfigure}
	\begin{subfigure}{.5\textwidth}
		\centering
		\includegraphics[scale=.55]{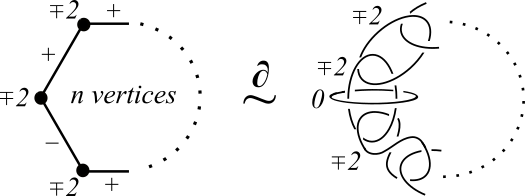}
		\caption{Another cyclic plumbing with boundary $\textbf{T}_{\pm n}$}\label{par1}
		\vspace{.2cm}
	\end{subfigure}

	\begin{subfigure}{\textwidth}
		\centering
		\includegraphics[scale=.55]{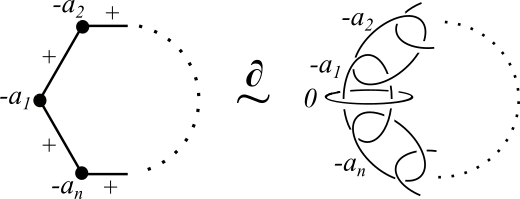}
		\caption{Cyclic plumbing with boundary $\textbf{T}_{A(\textbf{a})}$, where $\textbf{a}=(a_1,\ldots,a_n)$}\label{hypsurgery}
	\end{subfigure}
	\caption{The boundaries of cyclic plumbings are $T^2$-bundles over $S^1$}\label{cyclic}
\end{figure}

Up to conjugation, negative parabolic torus bundles have monodromies of the form $-T^{\pm n}$, where $n\ge0$ is an integer; we denote them by $\textbf{T}_{\pm n}$. $\textbf{T}_{\pm n}$ is the boundary of the cyclic plumbing shown in the left of Figure \ref{par0}. This plumbing diagram gives rise to the obvious surgery diagram of $\textbf{T}_{\pm n}$ shown in the right of Figure \ref{par0}. If $n\ge 2$, then $\textbf{T}_{\pm n}$ also bounds the negative/positive definite cyclic plumbing shown in the left of Figure \ref{par1}. This can be seen by performing blowups and blowdowns to the surgery diagram in Figure \ref{par0} to obtain the the surgery diagram in the right of Figure \ref{par1}, which is the boundary of the negative cyclic plumbing in the left of Figure \ref{par1}. Up to conjugation, positive hyperbolic torus bundles have monodromies of the form $ T^{-a_1}S\cdots T^{-a_n}S$, where $a_i\ge 2$ for all $i$ and $a_j\ge 3$ for some $j$. These are the boundaries of the positive cyclic plumbings shown in the left of Figure \ref{hypsurgery}. To simplify notation, we will use $\textbf{T}_{A(\textbf{a})}$ to denote the hyperbolic torus bundle with monodromy $A(\textbf{a})=T^{-a_1}S\cdots T^{-a_n}S$, where $\textbf{a}=(a_1,\ldots,a_n)$. For details, see \cite{neumann}.

The next two results provide two infinite families of torus bundles over $S^1$ that bound $\QQ S^1\times B^3s$. The first follows from a rather simple observation; we label it as a lemma for easy reference.

\begin{lem} All negative parabolic torus bundles bound $\QQ S^1\times B^3s$. \label{parlem}\end{lem}

The second infinite family consists of positive hyperbolic torus bundles that bound $\QQ S^1\times B^3s$. These $\QQ S^1\times B^3s$ are more difficult to construct. In particular, we will show that these $\QQ S^1\times B^3s$ never admit handlebody decompositions without 3-handles, unlike the $\QQ S^1\times B^3s$ that will be constructed in the proof of Lemma \ref{parlem}. For ease of notation, we use $-(a_1,\ldots,a_n)$ to denote the string $(-a_1,\ldots,-a_n)$, and a string of the form $(\ldots, a^{[x]},\ldots)$ denotes the string $(\ldots,\underbrace{a,\ldots,a}_{x},\ldots)$.

\begin{thm} Let $\textbf{a}=(3+x_1,2^{[x_2]},\ldots,3+x_{2k+1},2^{[x_1]}, 3+x_2,2^{[x_3]},\ldots,3+x_{2k},2^{[x_{2k+1}]})$, where $k\ge0$ and $x_i\ge0$ for all $i$. Then $\textbf{T}_{A(\textbf{a})}$ bounds a rational homology circle $W$ with $H^3(W)=\ZZ_2$. Moreover, any handlebody decomposition of any $\QQ S^1\times B^3$ bounded by $\textbf{T}_{A(\textbf{a})}$ necessarily contains 3-handles.\label{hypthm}\end{thm}

As an application of Lemma \ref{rationalballobstruction}, Lemma \ref{parlem}, and Theorem \ref{hypthm}, we can construct more infinite families of $\QQ S^3s$ bounding $\QQ B^4s$ (most of which are not $\QQ S^3s$). The following corollary highlights some such families. Many more examples can be constructed similarly.

\begin{figure}[t]
	\centering
	\begin{subfigure}{.2\textwidth}
		\centering
		\hspace{-.5cm}
		\includegraphics[scale=.65]{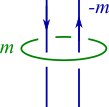}
		\caption{}\label{bounds2}
	\end{subfigure}
	\begin{subfigure}{.25\textwidth}
		\centering
		\hspace{-.5cm}
		\includegraphics[scale=.55]{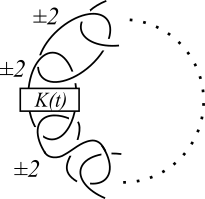}
		\caption{}\label{bounds4}
	\end{subfigure}
	\begin{subfigure}{.25\textwidth}
		\centering
		\hspace{-.5cm}
		\includegraphics[scale=.55]{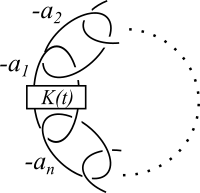}
		\caption{}\label{bounds3}
	\end{subfigure}
	\caption{$\QQ S^3s$ that bound $\QQ B^4s$. The boxes labeled $K(t)$ in (A), (C), and (D) indicate tying the strands passing through the box into a knot $K$ and adding $t$ full-twists. The link in (B) is a two-component link consisting of unknots with linking number 0; the blue vertical strands form a portion of one of the link components. The surgery coefficients in (C) are of the form\\ $-(3+x_1,2^{[x_2]},3+x_3,2^{[x_4]},\ldots,3+x_{2k+1},2^{[x_1]},  3+x_2,2^{[x_3]},\ldots,3+x_{2k},2^{[x_{2k+1}]})$. }\label{bounds}
\end{figure}

\begin{cor} Let $n,m,k,x_1,\ldots,x_{2k+1}\in\ZZ$ such that $n\ge2$, $m\neq0$, and either $k\ge 1$ or $x_1\ge 1$. Then the following $\QQ S^3s$ bound $\QQ B^4s$.
	\begin{enumerate}[(a)]
		\item $(m,-m)$-surgery along any two-component link with unknot components that have linking number 0 as in Figure \ref{bounds2} bounds a $\QQ B^4$.
		\item $\pm(2^{[n]})$-surgery along any link of the form shown in Figure \ref{bounds4} bounds a $\QQ B^4$.
		\item $-(3+x_1,2^{[x_2]},3+x_3,2^{[x_4]},\ldots,3+x_{2k+1},2^{[x_1]},  3+x_2,2^{[x_3]},\ldots,3+x_{2k},2^{[x_{2k+1}]})$-surgery along any link of the form shown in Figure \ref{bounds3} bounds a $\QQ B^4$.
	\end{enumerate}
	\label{rationalballcor}
\end{cor}

A natural question is whether there are any other torus bundles that bound $\QQ S^1\times B^3s$. This question is explored by the author in \cite{simone4}, where it is shown that there are no other torus bundles that bound $\QQ S^1\times B^3s$. Thus the only torus bundles that bound $\QQ S^1\times B^3s$ are those listed in Lemma \ref{parlem} and Theorem \ref{hypthm}. If we expand our view, however, and consider more general plumbed 3-manifolds with a single cycle, we can construct a large class of plumbed 3-manifolds that bound $\QQ S^1\times B^3s$ by using the $\textit{self-join operation}$.

\begin{definition}[c.f. Aceto \cite{aceto}] Let $X$ be a plumbing whose associated graph is a tree and let $v_1$ and $v_2$ be distinguished vertices.  Let $X_{v_1=\pm v_2}$ be the positive/negative plumbing obtained by identifying $v_1$ and $v_2$ and taking the sum of the corresponding weights to be the weight of the new vertex. We say that $X_{v_1=\pm v_2}$ is obtained from $X$ by \textit{self-joining} $X$ along $v_1$ and $v_2$.\label{definition}\end{definition}

\begin{prop} Let $X$ be a plumbing tree such that $Y=\partial X$ bounds a $\mathbb{Q}S^1\times B^3$. Let $v_1$ and $v_2$ be distinct vertices of $X$ and let $Q_{\pm}$ denote the intersection form of the plumbing $X_{v_1=\pm v_2}$. If $\det Q_{\pm}\neq0$, then $\partial (X_{v_1=\pm v_2})$ bounds a $\mathbb{Q}S^1\times B^3$.\label{constructionprop}\end{prop}

We will see that Proposition \ref{constructionprop} can be used to prove Lemma \ref{parlem}, but it cannot be used to prove Theorem \ref{hypthm}. Thus the plumbed 3-manifolds built using Proposition \ref{constructionprop} are not the only ones that bound $\QQ S^1\times B^3s$. Moreover, Lemma \ref{rationalballobstruction} and Proposition \ref{constructionprop} can be used to construct more infinite families of $\QQ S^3s$ bounding $\QQ B^4s$ (c.f. Corollary \ref{rationalballcor}).

\subsection{Symplectic Question} Let $\textbf{T}_{A(\textbf{a})}$ be a hyperbolic torus bundle listed in Theorem \ref{hypthm} and let $\xi$ be a tight contact structure on $\textbf{T}_{A(\textbf{a})}$. The question of whether $(\textbf{T}_{A(\textbf{a})},\xi)$ has a strong symplectic $\QQ S^1\times B^3$ filling is also of considerable interest. By a recent result of Christian \cite{christian}, there is no such filling when $\xi$ is virtually overtwisted. However, the question is more interesting when $\xi$ is universally tight.

Let $P_{\textbf{a}}$ be the negative-definite plumbing whose boundary is $\textbf{T}_{A(\textbf{a})}$ (Figure \ref{hypsurgery}). Then $P_{\textbf{a}}$ can be realized as the resolution of an isolated complex surface singularity and it thus admits a unique Milnor fillable contact structure $\xi_{can}$ (see, for example, \cite{neumann}), which is automatically universally tight by \cite{lekiliozbag}. 
In \cite{gaystipsiczsympsurg}, it is shown that if $P_{\textbf{a}}$ is embedded in an ambient symplectic 4-manifold $(X,\omega)$ such that the base spheres are symplectic and intersect $\omega-$orthogonally, then $P_{\textbf{a}}$ admits a symplectic structure with strongly convex boundary and the induced contact structure on $\textbf{T}_{A(\textbf{a})}$ is $\xi_{can}$. Thus, if $(\textbf{T}_{A(\textbf{a})},\xi_{can})$ has a strong symplectic $\QQ S^1\times B^3$ filling $W$, then by \cite{etnyreconvexity} one can excise $P_{ \textbf{a}}$ from $X$ and glue $W$ in its place to obtain a potentially small exotic 4-manifold. 

\begin{question} Does $(\textbf{T}_{A(\textbf{a})},\xi_{can})$ the have a strong $\QQ S^1\times B^3s$ filling? \label{question} \end{question}

\textit{Symplectic cut-and-paste} operations, in which plumbing trees are symplectically replaced by smaller 4-manifolds, have been successfully used to construct small exotic 4-manifolds (e.g. the rational blowdown \cite{fintushelstern}, star surgery \cite{karakurtstarkston}, 2-replaceability \cite{simone1}, etc). In these constructions, the replacement manifolds admit Stein structures that induce the correct contact structures; thus the replacements naturally lend themselves to symplectic cut-and-paste. By Theorem \ref{hypthm}, however, every $\QQ S^1\times B^3$ bounded by $\textbf{T}_{A(\textbf{a})}$ necessarily contains 3-handles. Consequently, $\textbf{T}_{A(\textbf{a})}$ does not bound a Stein $\QQ S^1\times B^3$, making an affirmative answer to Question \ref{question} more interesting, albeit, more difficult to prove.

\subsection{Organization} In Section \ref{spheres}, we will prove Lemma \ref{rationalballobstruction} and use it to prove Theorem \ref{thm1} and Corollary \ref{rationalballcor}. In Section \ref{torusbundles}, we will prove Lemma \ref{parlem} and Theorem \ref{hypthm} by explicitly constructing $\QQ S^1\times B^3s$ bounded by torus bundles.  Finally, in Section \ref{qcircles}, we will prove Proposition \ref{constructionprop} and use it to construct some more examples of plumbed 3-manifolds that bound $\QQ S^1\times B^3s$.

\subsection{Acknowledgments} Thanks to Marco Golla for pointing out that part of Theorem \ref{thm1} follows from Corollary \ref{rationalballcor}(a), and to O\u{g}uz \c{S}avk for sharing important references regarding rationally slice knots and Brieskorn spheres bounding rational homology balls.

\section{Rational Homology Spheres bounding rational homology balls}\label{spheres}

By Theorem \ref{hypthm}, the hyperbolic torus bundle with monodromy $T^{-3}S$ bounds a rational homology circle. It will be shown in Proposition \ref{hypprop} that this hyperbolic torus bundle can be realized as $0$-surgery on the figure-eight knot, which is known to bound a rational homology circle. As mentioned in the introduction, Akbulut-Larson used this fact in \cite{akbulutlarson} to construct infinite families of Brieskorn spheres that bound $\QQ B^4s$. In a similar vein, we will use this torus bundle, along with the broader classes of torus bundles of Theorem \ref{hypthm} to prove Theorem \ref{thm1} and Corollary \ref{rationalballcor}. This construction relies Lemma \ref{rationalballobstruction}, which we now recall and prove.\\

\noindent\textbf{Lemma \ref{rationalballobstruction}.} \textit{Let $Y$ bound a $\mathbb{Q}S^1\times B^3$, $W$, and let $K$ be a knot in $Y$ such that $[K]$ has infinite order in $H_1(Y;\mathbb{Z})$. Then the 4-manifold obtained by attaching a 2-handle to $W$ along $K$ is a $\mathbb{Q}B^4$. Consequently, any integer surgery along $K$ yields a $\QQ S^3$ that bounds a $\QQ B^4$.}

\begin{proof} Let $B$ be the 4-manifold obtained by attaching a 2-handle $h$ to $W$ along $K$. Since $H_2(Y,W)$ is a torsion group, $[K]\in H_1(Y)$ must map to an infinite order element $m\in H_1(W)$ under the map induced by inclusion. $K$ also represents a generator of $H_1(W\cap h)=H_1(S^1\times D^2)=\ZZ$ that maps to $m\in H_1(W)$ under the map induced by inclusion ($W\cap h$ is the attaching region of $h$ in $Y$). Now by considering the Mayer-Vietoris sequence for $(B,W, h)$, it is easy to see that $B$ must have the rational homology of the 4-ball. 
\end{proof}


For simplicity, we combine the cases in Theorem \ref{thm1} and Corollary \ref{rationalballcor} in the following theorem.

\begin{thm} Let $n,m,k,x_i\in\ZZ$ for all $1\le i\le 2k+1$ such that $n\ge2$, $m\neq0$, and either $k\ge 1$ or $x_1\ge 1$. Then the following $\QQ S^3s$ bound $\QQ B^4s$.
	\begin{enumerate}[(a)]
		\item  If $K$ is an unknotting number one knot with a positive (resp. negative) crossing that can be switched to unknot $K$, then $S^3_{-1}(K)$ (resp. $S^3_1(K)$) bounds a $\QQ B^4$.
		\item $(m,-m)$-surgery along any two-component link with unknot components that have linking number 0 as in Figure \ref{bounds2} bounds a $\QQ B^4$. 
		\item $\pm(2^{[n]})$-surgery along any link of the form shown in Figure \ref{bounds4} bounds a $\QQ B^4$.
		\item $-(3+x_1,2^{[x_2]},3+x_3,2^{[x_4]},\ldots,3+x_{2k+1},2^{[x_1]},  3+x_2,2^{[x_3]},\ldots,3+x_{2k},2^{[x_{2k+1}]})$-surgery along any link of the form shown in Figure \ref{bounds3} bounds a $\QQ B^4$.
\end{enumerate}
\label{thmcombined}
\end{thm}

\begin{proof} We start by proving part $(c)$. The proof of $(d)$ is identical. 
	Consider the negative parabolic torus bundle $\textbf{T}_{\pm n}$. By Lemma \ref{parlem},  $\textbf{T}_{\pm n}$ bounds a rational homology circle $W$. Let $K$ be a knot in $\textbf{T}_{\pm n}$ as depicted in the leftmost diagram of Figure \ref{attachtoball}. If we perform $n$-surgery along $K$ for any $n\in\ZZ$,  then by Lemma \ref{rationalballobstruction}, the resulting $\QQ S^3$, $Y$,  bounds a $\QQ B^4$. Starting with the surgery diagram for $Y$ shown on the left of Figure \ref{attachtoball}, slide the two strands of the $-1$-framed unknot passing through the 0-framed unknot over the knot $K$. Then cancel the 0-framed unknot and the knot $K$ to obtain the surgery diagram in the right of Figure \ref{attachtoball}.
	
	\begin{figure}
		\centering
		\includegraphics[scale=.6]{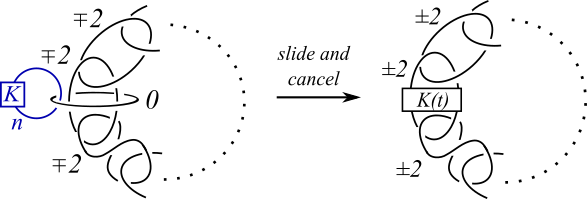}
		\caption{Performing $n$-surgery on $\textbf{T}_{\pm n}$ along the blue knot $K$ yields a $\QQ S^3$ that bounds a $\QQ B^4$.}
		\label{attachtoball}
	\end{figure}

	We now proceed to part (b). Let $\textbf{a}=(m+2,2^{[m-1]})$. If $m\ge 2$, then $\textbf{T}_{A(\textbf{a})}$ has the obvious surgery diagram shown in the left of Figure \ref{2comps}. After blowing up once, blowing down $m-1$ times, and isotoping, as in Figure \ref{2comps}, we obtain the surgery diagram on the right side of the figure. If $m=1$, then $\textbf{T}_{A(3)}$ has the surgery diagram shown in the left of Figure \ref{2compsa}. After blowing up and isotoping as in Figure \ref{2compsa}, we obtain the surgery on the right side of the figure. Thus the rightmost diagrams in Figures \ref{2comps} and \ref{2compsa} provide alternate surgery diagrams for $\textbf{T}_{A(\textbf{a})}$, where $\textbf{a}=(m+2,2^{[m-1]})$ and $m\ge 1$. 
	
	Let $Y$ be the $\QQ S^3$ obtained by $(m,-m)$-surgery along any two-component link with unknot components that have linking number 0 as in Figure \ref{bounds2}. Then the link can be isotoped so that $Y$ has the surgery diagram given in the top-left of Figure \ref{2comps2}, where the gray box contains the complexity of the second unknot (i.e. all of the crossings). Our goal is to show that $Y$ can be realized as $0$-surgery along a knot in $\textbf{T}_{A(\textbf{a})}$ that represents an infinite order element in $H_1(\textbf{T}_{A(\textbf{a})};\ZZ)$. In light of Theorem \ref{hypthm} and Lemma \ref{rationalballobstruction}, it will follow that $Y$ bounds a $\QQ B^4$.
	
		\begin{figure}[t]
		\centering
		\begin{subfigure}{\textwidth}
			\centering
			\vspace{.2cm}
			\includegraphics[scale=.6]{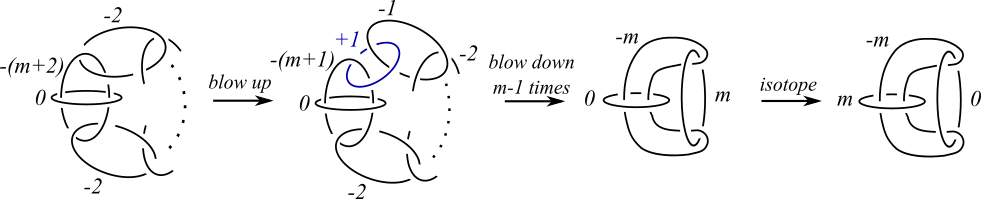}
			\caption{Surgery diagrams for $\textbf{T}_{A(\textbf{a})}$, where $\textbf{a}=(m+2,2^{[m-1]})$ and $m\ge 2$.} \label{2comps}
			\vspace{.5cm}
		\end{subfigure}
		\begin{subfigure}{\textwidth}
		\centering
		\vspace{.2cm}
		\includegraphics[scale=.6]{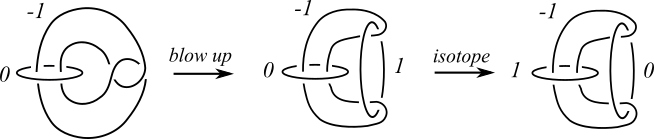}
		\caption{Surgery diagrams for $\textbf{T}_{A(3)}$.} \label{2compsa}
		\vspace{.5cm}
	\end{subfigure}
		\begin{subfigure}{\textwidth}
			\centering
			\vspace{.4cm}
			\includegraphics[scale=.65]{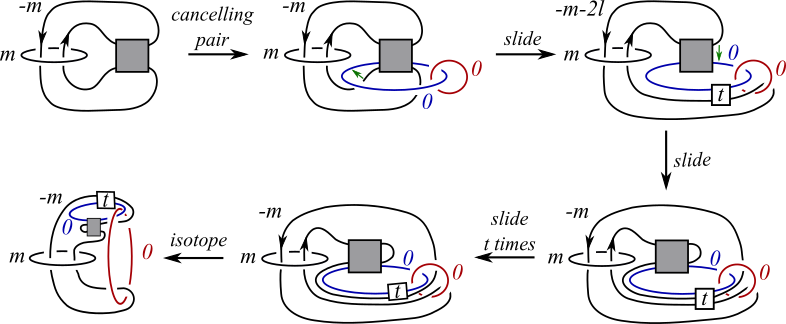}
			\caption{Showing that $Y$ is the boundary of a $\QQ B^4$} \label{2comps2}
		\end{subfigure}
		\caption{Proving $(-m,m)$-surgery along certain 2-component links (see Corollary \ref{rationalballcor}) bound $\QQ B^4s$.}
	\end{figure}
	
	Consider the surgery diagram of $Y$ shown in the top-left of Figure \ref{2comps2}. Let $L_1$ denote the $m$-framed unknot and let $L_2$ denote the complicated $-m$-framed unknot. The gray box contains two arcs that are arbitrarily knotted in a way yielding an unknot. Since $L_2$ is connected and the linking number of $L_1$ and $L_2$ is 0, the arc that begins near the bottom-right corner of the gray box must end near the bottom-left corner of the gray box; we will refer to this arc at the \textit{first arc}. The \textit{second arc} is the arc beginning at near the top-left corner and ending near the top-right corner. 
	
	We now add a canceling pair of 0-framed unknots to the surgery diagram as in the next diagram of Figure \ref{2comps2}. The blue unknot travels through the gray box parallel to the first arc such that it can be identified with the blackboard framing of the first arc. Pick an orientation for the blue unknot and let $l$ be the linking number of the blue unknot with $L_2$ and let $-t$ denote the writhe of the blue unknot (which equals the writhe of the first arc). Slide $L_2$ over the blue unknot, as indicated by the green arrow in the top-center diagram of Figure \ref{2comps2} to obtain the next diagram in Figure \ref{2comps2}, where the box labeled $t$ indicates $t$ full twists. Next, slide $L_2$ over the blue unknot again, as indicated by the second green arrow, to obtain the surgery diagram in the bottom-right of Figure \ref{2comps2}. Note that there are once again two arcs of $L_2$ passing through the gray box, which can be viewed as copies of the original two arcs. Thus $L_2$ is still an unknot and, moreover, the arc formed by the two arcs in the gray box along with the vertical strand connecting the endpoints of the arcs on the right side of the gray box can be isotoped to be an arc with no over- or under-crossings. 
	
	There are two strands of $L_2$ passing through the red unknot with opposite orientation. Slide the lower strand over the red unknot $t$ times to obtain the diagram in the bottom-center, which is isotopic to the diagram in the bottom-left of Figure \ref{2comps2}. Notice that without the blue unknot, this last surgery diagram is the surgery diagram of $\textbf{T}_{A(\textbf{a})}$ we found in Figure \ref{2comps}. Thus $Y$ is obtained from $\textbf{T}_{A(\textbf{a})}$ by performing 0-surgery along a knot with infinite order in $H_1(\textbf{T}_{A(\textbf{a})};\ZZ)$. Since $\textbf{T}_{A(\textbf{a})}$ bounds a $\QQ S^1\times B^3$, $Y$ bounds a $\QQ B^4$ by Lemma \ref{rationalballobstruction}.
	
	Finally, part (a) follows from part (b). Let $K$ be an unknotting number one knot with a positive crossing that can be changed to unknot $K$ and consider $S^3_{-1}(K)$. Blow up the positive crossing with a $+1$-framed unknot to obtain $(-1,+1)$-surgery along a two component link satisfying the conditions of part (b) (c.f. the top-left diagram in Figure \ref{2comps2}); it follows that $S^3_{-1}(K)$ bounds a $\QQ B^4$. Similarly, if $K$ is an unknotting number one knot with a negative crossing that can be changed to unknot $K$, then $S^3_{1}(K)$ bounds a $\QQ B^4$.
\end{proof}

\begin{remark} The $\QQ B^4s$ constructed in parts $(c)$ and $(d)$ of Corollary \ref{rationalballcor} are constructed by attaching a 2-handle along a knot that is homologous to the meridian of the $0$-framed unknot in the surgery diagram of some torus bundle. The $\QQ B^4s$ constructed in parts $(a)$ and $(b)$ rely on 2-handles attached along knots in more interesting ways, relative to the surgery diagram of $\textbf{T}_{A(\textbf{a})}$. These are just a few examples of infinite families of $\QQ S^3s$ that can be shown to bound $\QQ B^4s$. Using these techniques, one can construct many more, potentially interesting, examples.\end{remark}

The following proposition finishes the proof of Theorem \ref{thm1}.

\begin{prop} Let $K_n$ be a twist knot. If $n$ is odd, then $S^3_{-1}(K_n)$ has nontrivial Rohlin invariant. Consequently, $S^3_{-1}(K_n)$ is a nontrivial element in $\text{ker}(\Theta_{\ZZ}^3\to\Theta_{\QQ}^3)$.\label{kernelprop}\end{prop}

\begin{proof}
	Note that $K_n$ has unknotting number one and has a positive crossing that can be changed to unknot it, for all $n$. Thus by Theorem \ref{thmcombined}(a),  $S^3_{-1}(K_n)$ bounds a $\QQ B^4$ for all $n$. To show $S^3_{-1}(K_n)$ has nontrivial Rohlin invariant $\mu$ for odd $n$, we simply use Kirby Calculus to find a spin 4-manifold bounded by $S^3_{-1}(K_n)$ and calculate its signature. Figure \ref{rohlin} shows that $S^3_{-1}(K_n)$ can be described by surgery along a link with even surgery coefficients, when $n$ is odd. Let $X$ be the 4-manifold with handlebody diagram given by the last surgery diagram in Figure \ref{rohlin}. Since $X$ is built out of 2-handles with even framings, $X$ is spin (Proposition 5.7.1 in \cite{stipgompf}). Moreover, it is easy to check that $\sigma(X)=8$. Thus $\mu(S^3_{-1}(K_n))=1$.
	
		\begin{figure}
		\centering
		\includegraphics[scale=.6]{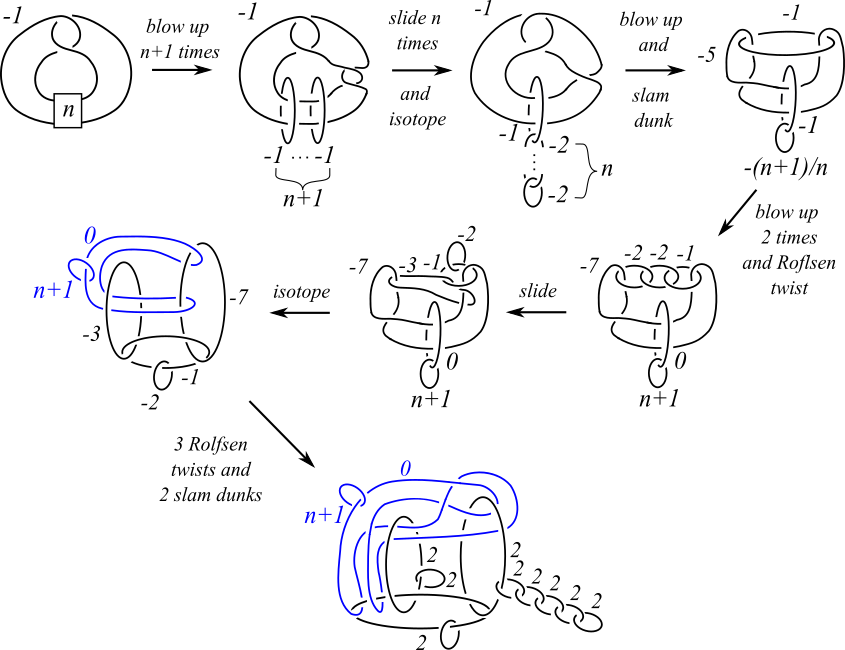}
		\caption{For $n$ odd, $S^3_{-1}(K_n)$ bounds a spin 4-manifold $X$ with $\sigma(X)=8$.}
		\label{rohlin}
	\end{figure}

\end{proof}

\section{Torus bundles bounding rational homology circles}\label{torusbundles}

In this section, we will prove Lemma \ref{parlem} and Theorem \ref{hypthm}. We start with the former, which is a simple observation. \\

\noindent \textbf{Lemma \ref{parlem}.} \textit{All negative hyperbolic torus bundles bound $\QQ S^1\times B^3s$.}

\begin{proof} Consider the obvious handlebody diagram of the cyclic plumbing bounded by the negative parabolic torus bundle with monodromy $-T^n$, show in the left of Figure \ref{parabolicrationalcircles}. After surgering, as in Figure \ref{parabolicrationalcircles}, we obtain the handlebody diagram of a 4-manifold whose boundary is the same parabolic torus bundle. A quick homology calculation shows that the new 4-manifold has the rational homology of $S^1\times B^3$.\end{proof}

\begin{figure}[h!]
	\centering
	\includegraphics[scale=.55]{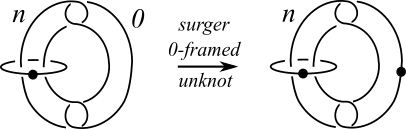}
	\caption{Negative parabolic torus bundles bound $\QQ S^1\times B^3s$.}\label{parabolicrationalcircles}
\end{figure}

We now turn our efforts to proving Theorem \ref{hypthm}, which we break into two propositions. We will first show that the hyperbolic torus bundles of Theorem \ref{hypthm} indeed bound $\QQ S^1\times B^3s$. Afterwards we will show that every handlebody decomposition of any $\QQ S^1\times B^3$ bounded by such a hyperbolic torus bundle necessarily contain 3-handles, unlike the $\QQ S^1\times B^3s$ bounded by the parabolic torus bundles of Lemma \ref{parlem}. Before diving into the proof of the first proposition, we need a quick definition and some background.

Let $(b_1,\ldots,b_k)$ be a string of integers such that $b_i\ge2$ for all $i$. If $b_j\ge 3$ for some $j$, then we can write this string in the form $(2^{[m_1]},3+n_1,\ldots,2^{[m_j]},2+n_j)$, where $m_i,n_i\ge 0$ for all $i$. The string $(c_1,\ldots,c_l)=(2+m_1, 2^{[n_1]},2+m_2,\ldots,3+m_j,2^{[n_j]})$ is called the \textit{dual string} of $(b_1,\ldots,b_k)$. If $b_i=2$ for all $1\le i\le k$, then we define its dual string to be $(k+1)$. As a topological interpretation of this, by Neumann \cite{neumann}, if $P$ is a linear plumbing of $D^2$-bundles over $S^2$ with Euler numbers $(b_1,\ldots,b_k)$, then the reversed-orientation plumbing $\overline{P}$ is a linear plumbing with Euler numbers $(c_1,\ldots,c_l)$. Moreover, the obvious handlebody diagram of $\overline{P}$ can be obtained from the obvious handlebody diagram of $P$ by performing suitable blowups and blowdowns. This procedure will be used in the proof of the following proposition.

\begin{prop} Let $\textbf{a}=(3+x_1,2^{[x_2]},\ldots,3+x_{2k+1},2^{[x_1]}, 3+x_2,2^{[x_3]},\ldots,3+x_{2k},2^{[x_{2k+1}]})$, where $k\ge0$ and $x_i\ge0$ for all $i$. Then $\textbf{T}_{A(\textbf{a})}$ bounds a rational homology circle $W$ with $H^3(W)=\ZZ_2$.\label{hypprop}\end{prop}

\begin{proof} 
	First suppose $\textbf{a}=(3)$. The left diagram in Figure \ref{figure8} is the obvious surgery diagram of $\textbf{T}_{A(3)}$ obtained from the plumbing diagram. By blowing down the $-1$-framed unknot, it is easy to see that $\textbf{T}_{A(3)}$ can realized as 0-surgery on the figure-eight knot. As mentioned in the introduction, it is well-known that 0-surgery on the figure-eight knot bounds a $\QQ S^1\times B^3$.

	Now assume $(k+\sum_{i=1}^{2k+1}x_i)\ge1$ and let $\textbf{a}=(3+x_1,2^{[x_2]},\ldots,3+x_{2k+1},2^{[x_1]}, 3+x_2,2^{[x_3]},\ldots,3+x_{2k},2^{[x_{2k+1}]})$. For simplicity, relabel the surgery coefficients $(-(3+x_1),-2^{[x_2]},\ldots,-(3+x_{2k+1}))$ by $(-(d_1+1),-d_2,\ldots,-d_{p-1},-(d_p+1))$. Then the coefficients $(2^{[x_1]},3+x_2,\ldots,2^{[x_{2k+1}]})$ are of the form $(e_1,\ldots, e_q)$, where $(d_1,\ldots, d_p)$ and $(e_1,\ldots,e_q)$ are dual strings. Consider the obvious surgery diagram for $\textbf{T}_{A(\textbf{a})}$ shown in Figure \ref{originalsurgery}. Blow up the linking of the $-(d_1+1)$- and $-e_q$-framed unknots with a $+1$-framed unknot and then consecutively blow down the $-1$-framed unknots. Continuing in this way---performing $+1$-blowups followed by $-1$-blowdowns---we will obtain a surgery diagram involving  $(-d_1,\ldots, -d_p, d_1,\ldots, d_p)$-surgery along a ``chain link" as in Figure \ref{newsurgery}. Let $L$ denote this chain link.
	
	\begin{figure}[t!]
		\centering
		\includegraphics[scale=.6]{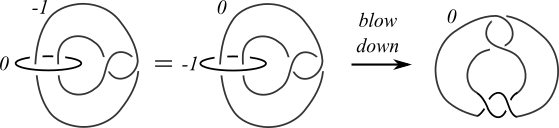}
		\caption{The hyperbolic torus bundle $\textbf{T}_{A(3)}$}\label{figure8}
	\end{figure}
	
	\begin{figure}
		\centering
		\begin{subfigure}{.45\textwidth}
			\centering
			\includegraphics[scale=.5]{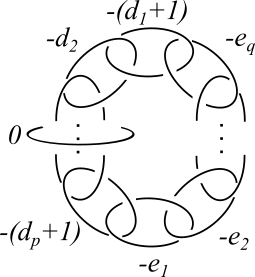}
			\caption{Relabeling the coefficients}\label{originalsurgery}
				\vspace{.3cm}
		\end{subfigure}
		\begin{subfigure}{.45\textwidth}
			\centering
			\includegraphics[scale=.5]{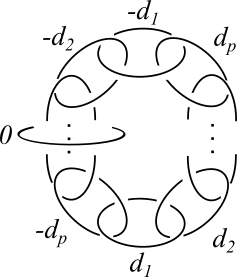}
			\caption{Alternate surgery diagram of $\textbf{T}_{A(\textbf{a})}$}\label{newsurgery}
				\vspace{.3cm}
		\end{subfigure}
		\begin{subfigure}{.9\textwidth}
			\centering
			\includegraphics[scale=.5]{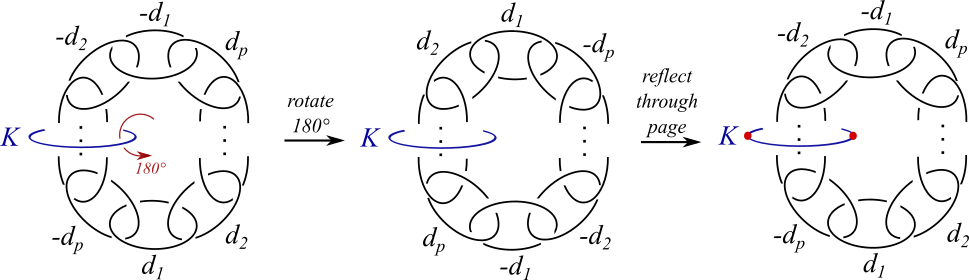}
			\caption{An orientation reversing involution $\tau'$ of $\textbf{T}'$ that maps $K$ to itself. The two red points on $K$ are the fixed points of $\tau'$.}\label{involution}
		\end{subfigure}
		\caption{The hyperbolic torus bundles $\textbf{T}_{A(\textbf{a})}$ with $\textbf{a}=(3+x_1,2^{[x_2]},\ldots,3+x_{2k+1},2^{[x_1]}, 3+x_2,2^{[x_3]},\ldots,3+x_{2k},2^{[x_{2k+1}]})$ admit orientation-reversing involutions and consequently bound $\QQ S^1\times B^3s$}
	\end{figure}
	
Let $S^3_{\textbf{d}}(L)$ be the 3-manifold obtained by $\textbf{d}=(-d_1,\ldots, -d_p, d_1,\ldots, d_p)$-surgery long $L$ and let $K\subset S^3_{\textbf{d}}(L)$ be the knot shown in the leftmost diagram in Figure \ref{involution}. We claim that there is an orientation-reversing involution $\tau$ of $S^3_{\textbf{d}}(L),$ fixing $K$, with fixed point set $S^0\subset K$. This is achieved by performing a $180^{\circ}$ rotation of $S^3$ about $K$ so that $K$ is fixed (as shown in Figure \ref{involution}) and then reflecting the chain link through the page so that $K$ maps to itself and $\tau$ has two fixed points, which both lie on $K$ (see Figure \ref{involution}). Since $K$ maps to itself under this involution, $K$ is, by definition, a strongly negative amphicheiral knot in $S^3_{\textbf{d}}(L)$. By Lemma 2.3 in \cite{kawauchinegamph}, performing 0-surgery along $K$ yields a $\QQ S^1\times S^2$ that bounds a rational homology circle $W$ such that $H^3(W;\ZZ)=\ZZ_2$. More precisely, $W$ is the mapping cylinder of $p: \textbf{T}_{A(\textbf{a})}\to\textbf{T}_{A(\textbf{a})}/\tau'$, where $\tau'$ is the fixed point free orientation reversing involution of $\textbf{T}_{A(\textbf{a})}$ inherited from $\tau$.
\end{proof}

Before proving that every handlebody decomposition of every $\QQ S^1\times B^3$ bounded by one of the hyperbolic torus bundles listed in Theorem \ref{hypthm} necessarily contains 3-handles (Proposition \ref{3handlesprop} below), we must first gather a few technical results.

\begin{lem} If $Y$ bounds a $\QQ S^1\times B^3$, then the torsion part of $H_1(Y;\ZZ)$ has square order.\label{squarelem}\end{lem}
\begin{proof}
	It is well-known that if a $\QQ S^3$ bounds a $\QQ B^4$, then its first homology group has square order (see, for example, Lemma 3 in \cite{cassongordon86}). The proof involves exploring the long exact sequence of the pair. An analogous argument shows that the same is true of $\QQ S^1\times S^2$s that bound $\QQ S^1\times B^3s$.
\end{proof}

\begin{lem}[Lemma 10 in \cite{sakuma}] $|\textup{Tor}(H_1(\textbf{T}_{ A(\textbf{a})};\ZZ))|=\textup{tr}(A(\textbf{a}))-2$.\label{order}
\end{lem}

%

\begin{cor} $|\textup{Tor}(H_1(\textbf{T}_{A(\textbf{a})^2};\ZZ))|$ is not a square for all $\textbf{a}$.\label{squarenotsquare}\end{cor}

\begin{proof}
Let $A=A(\textbf{a})=\begin{bmatrix} a & b\\c & d\end{bmatrix}$, where $ad-bc=1$. Then by Lemma \ref{order}, $|\textup{Tor}(H_1(\textbf{T}_{ A(\textbf{a})};\ZZ))|=a+d-2$. Thus $A^2=\begin{bmatrix} a^2+bc & ab+bd \\ ac+cd & bc+d^2\end{bmatrix}$. Once again, by Lemma \ref{order}, $|\textup{Tor}(H_1(\textbf{T}_{ A(\textbf{a})^2};\ZZ))|=a^2+2bc+d^2-2=a^2+d^2+2(ad-1)-2=(a+d)^2-4$. Since $\textbf{T}_{A(\textbf{a})}$ is positive hyperbolic, $a+d>2$. Thus $(a+d)^2-4$ is not a square.
\end{proof}

\begin{lem} Let $W$ be a $\QQ S^1\times B^3$ admitting a handlebody decomposition without 3-handles and let $\widetilde{W}$ be an n-fold cover of $W$. If $\partial \widetilde{W}$ is a $\QQ S^1\times S^2$, then $\widetilde{W}$ is a $\QQ S^1\times B^3$.\label{coverlem}\end{lem}

\begin{proof} Let $Y=\partial W$ and $\widetilde{Y}=\partial \widetilde{W}$. Since $W$ admits a handlebody decomposition without 3-handles, $\widetilde{W}$ also admits a handlebody decomposition without 3-handles. Thus $H_3(\widetilde{W};\ZZ)=0$. By Poincar{\'e} duality and the Universal Coefficient Theorem, we have that
	$$H_1(\widetilde{W},\widetilde{Y};\QQ)\cong H^3(\widetilde{W};\QQ)\cong \text{Hom}(H_3(\widetilde{W};\ZZ),\QQ)\oplus\text{Ext}(H_2(\widetilde{W};\ZZ),\QQ)=0.$$
	Thus the map $H_1(\widetilde{Y};\QQ)\to H_1(\widetilde{W};\QQ)$ induced by inclusion is surjective. Since $\widetilde{Y}$ is a $\QQ S^1\times S^2$, it follows that $\text{rank}(H_1(\widetilde{W};\QQ))\le 1$. Finally, since $\chi(\widetilde{W})=p\chi(W)=0$ and  $H_3(\widetilde{W};\QQ)=0$, we necessarily have that $H_1(\widetilde{W};\QQ)=\QQ$ and $H_2(\widetilde{W};\QQ)=0$. \end{proof}

\noindent Equipped with the above lemmas and corollary, we are now ready to finish the proof of Theorem \ref{hypthm}.

\begin{prop} Let $W$ be a $\QQ S^1\times B^3$ bounded by a hyperbolic torus bundle $\textbf{T}_{A(\textbf{a})}$. Then every handlebody decomposition of $W$ necessarily contains 3-handles.\label{3handlesprop}\end{prop}

\begin{proof}
	 Suppose that $\textbf{T}_{A(\textbf{a})}$ bounds a $\QQ S^1\times B^3$, $W$, admitting a handlebody decomposition without 3-handles. Consider the obvious surgery diagram of $\textbf{T}_{ A(\textbf{a})}$ as in Figure \ref{hypsurgery}. Let $\mu_i$ denote the homology class of the meridian of the $-a_i$-framed surgery curve and let $\mu_0$ denote the homology class of the meridian of the 0-framed surgery curve. Then $H_1(\textbf{T}_{A(\textbf{a})};\ZZ)$ is generated by $\mu_0,\ldots,\mu_n$. Consider the torus bundle $\textbf{T}_{ A(\textbf{a})^2}$, which has monodromy $(T^{-a_1}S\cdots T^{-a_n}S)^2$. There is an obvious $\ZZ_2$-action on $\textbf{T}_{ A(\textbf{a})^2}$; in the obvious surgery diagram of $\textbf{T}_{ A(\textbf{a})^2}$, rotate the chain link through the 0-framed unknot $180^{\circ}$ (c.f. Figure \ref{involution}). The quotient of $\textbf{T}_{ A(\textbf{a})^2}$ by this action is clearly $\textbf{T}_{A(\textbf{a})}$ and the induced map $f:H_1(\textbf{T}_{ A(\textbf{a})};\ZZ)\to \ZZ_2$ satisfies $f(\mu_0)=1$ and $f(\mu_i)=0$ for all $1\le i\le n$. 
	  
	  Let $i_*:H_1(\textbf{T}_{ A(\textbf{a})};\ZZ)\to H_1(W;\ZZ)$ be the map induced by inclusion. Since $W$ has no 3-handles, $i_*$ is surjective; hence there exists a basis for $H_1(W;\ZZ)$ of the form $\{m_0,m_1,\ldots,m_k\}$, where $m_0:=i_*(\mu_0)$ and $m_i$ is a torsion element for all $1\le i\le k$. Define $g:H_1(W;\ZZ)\to\ZZ_2$ by $g(m_0)=1$ and $g(m_i)=0$ for all $1\le i\le k$. Then $g$ is a surjective homomorphism satisfying $f=g\circ i_*$. Let $\widetilde{W}$ be the double cover of $W$ induced by $g$. Then $\partial \widetilde{W}=\textbf{T}_{A(\textbf{a})^2}$ and by Lemma \ref{coverlem}, $\widetilde{W}$ is a $\QQ S^1\times B^3$. But by Corollary \ref{squarenotsquare}, $|\text{Tor}(H_1(\textbf{T}_{(A(\textbf{a}))^2};\ZZ))|$ is not a square, which contradicts Lemma \ref{squarelem}. Thus $W$ must contain 3-handles. 
 \end{proof}

\begin{remark} The proof of Proposition \ref{3handlesprop} also holds for negative hyperbolic torus bundles. We left this case out to simplify the notation. In \cite{simone4}, it is shown that no negative hyperbolic torus bundle bounds a $\QQ S^1\times B^3$, so consideration of that case would be meaningless. \end{remark}

\section{Constructing plumbed 3-manifolds that bound rational homology circles}\label{qcircles}

In this section, we develop a method for constructing many plumbed 3-manifolds that bound $\QQ S^1\times B^3s$. Throughout this section, we will use the same notation to denote a plumbing and its associated graph. We begin by reviewing a useful construction by Aceto.

\begin{definition}[Aceto \cite{aceto}] Let $X_i$ be a plumbing tree with a distinguished vertex $v_i$, for $i=1,2$. Let $X$ be the plumbing tree obtained from $X_1$ and $X_2$ by identifying the two distinguished vertices and taking the sum of the corresponding weights to be the new weight (See Figure \ref{join}). We say that $X$ is obtained by \textit{joining} together $X_1$ and $X_2$ along $v_1$ and $v_2$ and we write $X=X_1\bigvee_{v_1=v_2}X_2$. We call this operation the \textit{join operation}.\end{definition}

 \begin{figure}[h!]
\centering
\begin{subfigure}[h]{.4\textwidth}
\centering
\includegraphics[scale=.45]{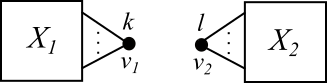}
\caption{Two plumbing trees $(X_i,v_i)$}\label{prejoin}
\end{subfigure}
\begin{subfigure}[h]{.4\textwidth}
\centering
\includegraphics[scale=.45]{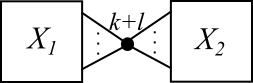}
\caption{The plumbing $X_1\bigvee_{v_1=v_2}X_2$}\label{postjoin} 
\end{subfigure}
\caption{Applying the join operation to the vertices $v_1$ and $v_2$}\label{join}
\end{figure}

On the 4-manifold level, consider the obvious handlebody diagram of $X_1\natural X_2$. Let $K_i$ denote the unknot to which the 2-handle associated with the vertex $v_i$ is attached. Let $U$ be an unknot such that $lk(K_i,U)=1$ for $i=1,2$ and such that there exists a sphere surrounding $U$ that intersects the handlebody diagram of $X_1\natural X_2$ in exactly four points, namely two points on $K_1$ and two points on $K_2$ (See Figure \ref{handlebodyjoin1}). Now attach a 0-framed 2-handle along $U$. By sliding $K_2$ over $K_1$, surgering $U$ into a dotted circle, and performing a handle cancellation, we obtain the obvious handlebody diagram of $X$. See Figure \ref{handlebodyjoin1a}). The following result gives a way to construct plumbed 3-manifolds whose graphs are trees that bound $\mathbb{Q}S^1\times B^3$s.

\begin{figure}[h]
\centering
\begin{subfigure}{\textwidth}
\centering
\captionsetup{width=.8\linewidth}
\includegraphics[scale=.45]{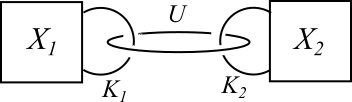}
\caption{Joining two plumbing trees to obtain $X=X_1\bigvee_{v_1=v_2}X_2$}\label{handlebodyjoin1}
\vspace{.3cm}
\end{subfigure}
\begin{subfigure}{\textwidth}
\centering
\captionsetup{width=.9\linewidth}
\includegraphics[scale=.45]{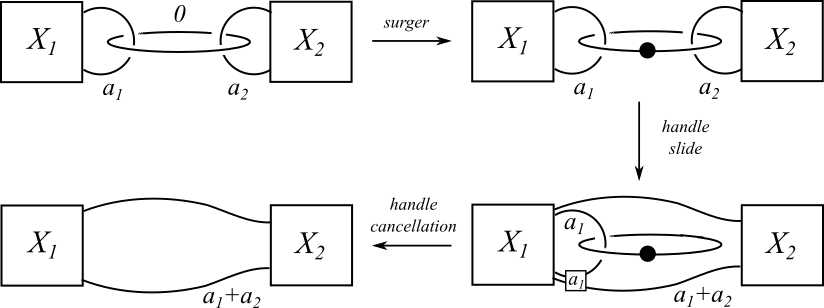}
\caption{Using Kirby calculus to obtain the obvious handlebody diagram of $X$.}\label{handlebodyjoin1a}
\end{subfigure}
\caption{Obtaining handlebody descriptions of the join operation}\label{handlebodies}
\end{figure}

\begin{prop}[Aceto \cite{aceto}] Let $(X, v)$ be a plumbing tree with distinguished vertex $v$ such that $\partial X=S^1\times S^2$ and $\partial (X\backslash v)$ is a $\mathbb{Q}S^3$. Let $(X',v')$ be a plumbing tree with distinguished vertex $v'$ such that $\partial X'$ is a $\mathbb{Q}S^1\times S^2$. If $\partial X'$ bounds a $\mathbb{Q}S^1\times B^3$, then so does $\partial (X \bigvee_{v=v'} X').$\label{acetoprop}\end{prop}

We now define the self-join operation, which will allow us to construct plumbed 3-manifolds with cycles. This was defined in the introduction, but for convenience, we recall it here.\\

\noindent\textbf{Definition \ref{definition}.} Let $X$ be a plumbing whose associated graph is a tree and let $v_1$ and $v_2$ be distinguished vertices.  Let $X_{v_1=\pm v_2}$ be the positive/negative plumbing obtained by identifying $v_1$ and $v_2$ and taking the sum of the corresponding weights to be the weight of the new vertex. We say that $X_{v_1=\pm v_2}$ is obtained from $X$ by \textit{self-joining} $X$ along $v_1$ and $v_2$.\\

On the 4-manifold level, we can once again consider the obvious handlebody diagram of $X$. Orient the attaching circles of the 2-handles so that all linking numbers of all adjacent unknots are $+1$. Let $K_i$ denote the unknot to which the 2-handle associated with the vertex $v_i$ is attached. Consider the obvious handlebody diagram for $X\natural(S^1\times B^3)$ obtained by adding a 1-handle to $X$. Now, as above, we can obtain $X_{v_1=\pm v_2}$ from $X\natural(S^1\times B^3)$ by adding a particular 0-framed 2-handle. Let $U_{\mp}$ be an unknot such that: $lk(K_1,U_{\mp})=1$ and $lk(K_2,U_{\mp})=\mp1$; there exists a sphere surrounding $U_{\mp}$ that intersects the handlebody diagram of $X\natural(S^1\times B^3)$ in precisely four points, namely two points on $K_1$ and two points on $K_2$; and $U_{\mp}$ ``passes through" the 1-handle (see Figure \ref{handlebodyjoin2}). Now attach a 0-framed 2-handle along $U_{\mp}$. As in the case with trees, by sliding $K_2$ over $K_1$, surgering $U_{\mp}$ into a dotted circle, and performing a handle cancellation, we obtain the obvious handlebody diagram of $X_{v_1=\pm v_2}$ (see Figures \ref{handlebodyjoin2a} and \ref{handlebodyjoin2b}). The following gives us a way to construct plumbed 3-manifolds with a single cycle that bound $\mathbb{Q}S^1\times B^3$s.

\begin{figure}[h]
\centering
\begin{subfigure}{.32\textwidth}
\centering
\captionsetup{width=.9\linewidth}
\includegraphics[scale=.45]{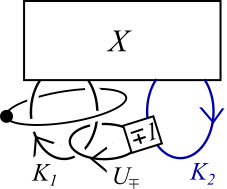}
\caption{Self-joining a plumbing tree to obtain $X_{v_1=\pm v_2}$}\label{handlebodyjoin2}
\end{subfigure}
\begin{subfigure}{.32\textwidth}
\centering
\captionsetup{width=.9\linewidth}
\includegraphics[scale=.45]{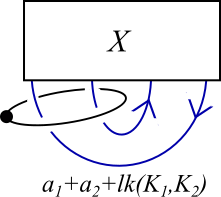}
\caption{Handlebody diagram of $X_{v_1=v_2}$}\label{handlebodyjoin2a}
\end{subfigure}
\begin{subfigure}{.32\textwidth}
\centering
\captionsetup{width=.9\linewidth}
\includegraphics[scale=.45]{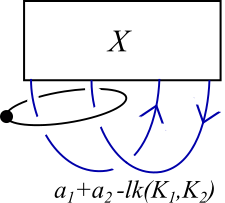}
\caption{Handlebody diagram of $X_{v_1=-v_2}$}\label{handlebodyjoin2b}
\end{subfigure}
\caption{Obtaining handlebody descriptions of the self-join operation}\label{handlebodies}
\end{figure}

\vspace{.3cm}

\noindent\textbf{Proposition \ref{constructionprop}.} \textit{Let $X$ be a plumbing tree such that $Y=\partial X$ bounds a $\mathbb{Q}S^1\times B^3$. Let $v_1$ and $v_2$ be distinct vertices of $X$ and let $Q_{\pm}$ denote the intersection form of the plumbing $X_{v_1=\pm v_2}$. If $\det Q_{\pm}\neq0$, then $\partial (X_{v_1=\pm v_2})$ bounds a $\mathbb{Q}S^1\times B^3$.}

\begin{proof}
Let $W$ be a $\QQ S^1\times B^3$ bounded by $Y$. First consider $X'=X\natural(S^1\times B^3)$ and note that $Y'=\partial X'=Y\#(S^1\times S^2)$ bounds a rational homology $(S^1\times B^3)\natural(S^1\times B^3)$, namely $W'=W\natural(S^1\times B^3)$. Consider the obvious surgery diagram of $Y'$ obtained as the boundary of the handlebody diagram of $X'$. By attaching a 0-framed 2-handle to $W'$ along an unknot $U_{\mp}$, as described in the paragraph preceding the statement of this proposition, we will obtain a 4-manifold $Z_{\pm}$ with boundary $\partial (X_{v_1=\pm v_2})$. We claim that $Z_{\pm}$ is a $\QQ S^1\times B^3$.

Since $Z_{\pm}$ is obtained by attaching a 1-handle and 2-handle to $W$, which is a $\QQ S^1\times B^3$, it follows that $\text{rank}H_1(Z_{\pm};\QQ)\in\{1,2\},$ $H_3(Z_{\pm};\QQ)=0$, and $\chi(Z_{\pm})=\chi(W)=0$. Thus, if we can show that $H_1(Z_{\pm};\QQ)=\QQ$, then since $\chi(Z_{\pm})=0$, it will follow that $H_2(Z_{\pm};\QQ)=0$, implying that $Z_{\pm}$ is a $\QQ S^1 \times B^3$.
	
Consider the obvious handlebody diagram of $X_{v_1=\pm v_2}$ as in Figure \ref{handlebodies}. Since the 1-handle has zero linking number with the attaching circles of every 2-handle, $Q_{\pm}$ is simply the linking matrix of the attaching circles of the 2-handles. Consider the obvious surgery diagram for $\partial (X_{v_1=\pm v_2})$ inherited from the plumbing $X_{v_1=\pm v_2}$. It is clear that $\partial (X_{v_1=\pm v_2})$ has linking matrix of the form $\begin{bmatrix}0&0\\0&Q_{\pm}\end{bmatrix}$ and $H_1(\partial (X_{v_1=\pm v_2});\ZZ)=\ZZ\oplus A$ for some abelian group $A$. Since $\det Q_{\pm}\neq0$, $A$ is necessarily finite. Now, since $H_1(W,Y;\QQ)=0$, it follows from the long exact sequence of the pair that $\text{rank}H_1(Z_{\pm};\QQ)=1$.
\end{proof}

\begin{example} 
\normalfont
 Let $X$ be the linear plumbing of length $n$ shown in Figure \ref{linear} and let $v_1$ and $v_n$ be the first and last vertices. Then $\partial X=S^1\times S^2$ and $X_{v_1=- v_2}$ is a cyclic plumbing whose boundary is $\textbf{T}_n$, the parabolic torus bundle with monodromy $-T^n$ (c.f. Figure \ref{par1}). Let $Q$ be the intersection form of $X_{v_1=-v_2}$. Then it is an $n\times n$ matrix of the form 
$$Q=\begin{bmatrix}
-2 & 1 & 0  &\cdots & 0&  -1 \\
1 & -2 & 1 &\cdots & 0& 0\\
\vdots & & \ddots & & & \vdots \\
\vdots & & & \ddots  & & \vdots \\
0 & 0 &\cdots & 1 & -2 & 1\\
-1 & 0 &  \cdots & 0 & 1 & -2
\end{bmatrix}$$
A quick calculation shows that, $|\det Q|=4$ (this can also be computed using Lemma \ref{order}). Thus by Proposition \ref{constructionprop}, parabolic torus bundles with monodromy $-T^n$ bound $\QQ S^1\times B^3s$ for all $n$. This gives an alternate proof of Lemma \ref{parlem}. 
 
 \begin{figure}
 	\centering
 	\begin{subfigure}{.45\textwidth}
 	\centering
 	\includegraphics[scale=.45]{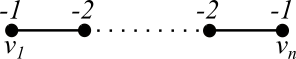}
 	\caption{Linear plumbing with boundary $S^1\times S^2$}\label{linear}
 	\end{subfigure}
 	\begin{subfigure}{.45\textwidth}
 	\centering
 	\includegraphics[scale=.45]{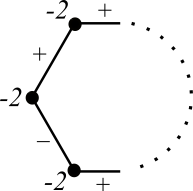}
 	\caption{Cyclic plumbing with boundary $\textbf{T}_n$}
 	\label{par}
 \end{subfigure}
\caption{Performing the self-join operation on $S^1\times S^2$ to obtain a negative parabolic torus bundle}
 \end{figure}
 
 The hyperbolic torus bundles $\textbf{T}_{A(\textbf{a})}$ of Theorem \ref{hypthm}, on the other hand, cannot be obtained using Proposition \ref{constructionprop}. This can be seen as follows. Suppose $\textbf{T}_{A(\textbf{a})}$ can be obtained by Proposition \ref{constructionprop} and let $\textbf{a}=(a_1,\ldots,a_n)$. Then we can undo the self-join operation to obtain a linear plumbed 3-manifold with the rational homology of $S^1\times S^2$ (in fact, since linear plumbed 3-manifolds are lens spaces, the only possibility is $S^1\times S^2$).
 Since $a_i\ge 2$ for all $i$ and $a_j\ge 3$ for some $j$, every internal vertex of the linear plumbing has weight at most $-2$ and either: there exists an internal vertex with weight at most $-3$; or one of the end vertices has weight at most $-2$. It is clear that such plumbed 3-manifolds are necessarily $\QQ S^3s$.
 \label{example1}
\end{example}

As the next example shows, we can also construct plumbed 3-manifolds bounding $\QQ S^1\times B^3s$ whose associated graphs are not cyclic.

\begin{example} 
\normalfont
 Consider the two plumbings $X_1$ and $X_2$ with distinguished vertices $v_1$ and $v_2$ depicted in Figure \ref{2trees}. Notice that $\partial X_1=\partial X_2=S^1\times S^2$ and $\partial(X_1\backslash v_1)=L(2,1)\#L(2,1)$ is a $\mathbb{Q}S^3$. Thus by Proposition \ref{acetoprop}, $\partial X=\partial (X_1\bigvee_{v_1=v_2}X_2)$ bounds a $\mathbb{Q}S^1\times B^3$. Now glue together the distinguished vertices $w_1$ and  $w_2$ of $X$ depicted in Figure \ref{1tree} to form $X_{w_1=\pm w_2}$ as shown in Figure \ref{acycle0}. Let $Q_{\pm}$ be the intersection form of $X_{w_1=\pm w_2}$. Calculations as in Example \ref{example1} show that $\det Q_{\pm}\neq 0$. Thus, by Proposition \ref{constructionprop}, the plumbed 3-manifold $\partial X_{w_1=\pm w_2}$ bounds a $\mathbb{Q}S^1\times B^3$.
 \label{example2}
\end{example}

\begin{figure}[h!]
\centering
\begin{subfigure}{\textwidth}
	\centering
	\includegraphics[scale=.5]{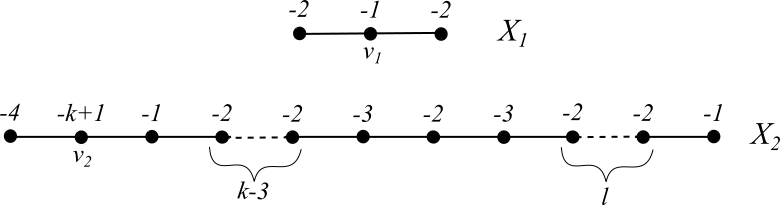}
	\caption{Two plumbing trees $(X_i,v_i)$}\label{2trees}
\end{subfigure}
\begin{subfigure}{\textwidth}
	\centering
	\captionsetup{width=.8\linewidth}
	\includegraphics[scale=.5]{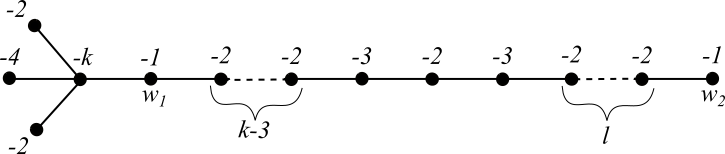}
	\caption{$X=X_1\bigvee_{v_1=v_2}X_2$. Consider the two distinguished vertices $w_1$ and $w_2$.}\label{1tree}
\end{subfigure}
\begin{subfigure}{\textwidth}
	\centering
	\captionsetup{width=.8\linewidth}
	\includegraphics[scale=.5]{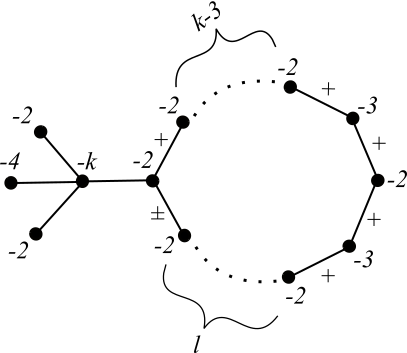}
	\caption{Self-joining the plumbing tree $X$ to obtain $X_{w_1=\pm w_2}$}\label{acycle0}
\end{subfigure}
\caption{Constructing a plumbed 3-manifold that bounds a $\QQ S^1\times B^3$.}\label{joinoperation}
\end{figure}

\bibliographystyle{plain}
\bibliography{Bibliography}

\begin{thebibliography}{10}

\bibitem{kirbyproblemlist}
Problems in low-dimensional topology.
\newblock In Rob Kirby, editor, {\em Geometric topology ({A}thens, {GA},
  1993)}, volume~2 of {\em AMS/IP Stud. Adv. Math.}, pages 35--473. Amer. Math.
  Soc., Providence, RI, 1997.

\bibitem{aceto}
Paolo Aceto.
\newblock Rational homology cobordisms of plumbed manifolds.
\newblock {\em Algebr. Geom. Topol.}, 20(3):1073--1126, 2020.

\bibitem{acetogolla}
Paolo Aceto and Marco Golla.
\newblock Dehn surgeries and rational homology balls.
\newblock {\em Algebr. Geom. Topol.}, 17(1):487--527, 2017.

\bibitem{akbulutlarson}
Selman Akbulut and Kyle Larson.
\newblock Brieskorn spheres bounding rational balls.
\newblock {\em Proc. Amer. Math. Soc.}, 146(4):1817--1824, 2018.

\bibitem{cassongordon86}
A.~J. Casson and C.~McA. Gordon.
\newblock Cobordism of classical knots.
\newblock In {\em \`A la recherche de la topologie perdue}, volume~62 of {\em
  Progr. Math.}, pages 181--199. Birkh\"{a}user Boston, Boston, MA, 1986.
\newblock With an appendix by P. M. Gilmer.

\bibitem{cassonharer}
Andrew Casson and John Harer.
\newblock Some homology lens spaces which bound rational homology balls.
\newblock {\em Pacific Journal of Mathematics}, 96(1):23--36, 1981.

\bibitem{christian}
Austin Christian.
\newblock On symplectic fillings of virtually overtwisted torus bundles.
\newblock {\em https://arxiv.org/abs/1909.01262}.

\bibitem{cochranetal}
Tim~D. Cochran, Bridget~D. Franklin, Matthew Hedden, and Peter~D. Horn.
\newblock Knot concordance and homology cobordism.
\newblock {\em Proc. Amer. Math. Soc.}, 141(6):2193--2208, 2013.

\bibitem{etnyreconvexity}
John~B. Etnyre.
\newblock Symplectic convexity in low-dimensional topology.
\newblock {\em Topology Appl.}, 88(1-2):3--25, 1998.
\newblock Symplectic, contact and low-dimensional topology (Athens, GA, 1996).

\bibitem{fickle}
Henry~Clay Fickle.
\newblock Knots, {${\bf Z}$}-homology {$3$}-spheres and contractible
  {$4$}-manifolds.
\newblock {\em Houston J. Math.}, 10(4):467--493, 1984.

\bibitem{fintushelsternmuinvt}
Ronald Fintushel and Ronald~J. Stern.
\newblock A {$\mu$}-invariant one homology {$3$}-sphere that bounds an
  orientable rational ball.
\newblock In {\em Four-manifold theory ({D}urham, {N}.{H}., 1982)}, volume~35
  of {\em Contemp. Math.}, pages 265--268. Amer. Math. Soc., Providence, RI,
  1984.

\bibitem{fintushelstern}
Ronald Fintushel and Ronald~J. Stern.
\newblock Rational blowdowns of smooth 4-manifolds.
\newblock {\em Journal of Differential Geometry}, 46(2):181--235, 1997.

\bibitem{gaystipsiczsympsurg}
David~T. Gay and Andr{\'a}s~I. Stipsicz.
\newblock Symplectic surgeries and normal surface singularities.
\newblock {\em Algebr. Geom. Topol.}, 9(4):2203--2223, 2009.

\bibitem{stipgompf}
Robert Gompf and Andr{\'a}s Stipsicz.
\newblock {\em 4-manifolds and {Kirby} Calculus}, volume~20 of {\em Graduate
  Studies in Mathematics}.
\newblock American Mathematical Society, Providence, RI, 1999.

\bibitem{karakurtstarkston}
\c{C}a\u{g}ri Karakurt and Laura Starkston.
\newblock Surgery along star-shaped plumbings and exotic smooth structures on
  4-manifolds.
\newblock {\em Algebr. Geom. Topol.}, 16(3):1585--1635, 2016.

\bibitem{kawauchinegamph}
Akio Kawauchi.
\newblock Rational-slice knots via strongly negative-amphicheiral knots.
\newblock {\em Commun. Math. Res.}, 25(2):177--192, 2009.

\bibitem{kimwurationalslice}
Min~Hoon Kim and Zhongtao Wu.
\newblock On rational sliceness of {M}iyazaki's fibered, -amphicheiral knots.
\newblock {\em Bull. Lond. Math. Soc.}, 50(3):462--476, 2018.

\bibitem{lecuonamontesinosknots}
Ana~G. Lecuona.
\newblock On the slice-ribbon conjecture for {M}ontesinos knots.
\newblock {\em Trans. Amer. Math. Soc.}, 364(1):233--285, 2012.

\bibitem{lekiliozbag}
Yanki Lekili and Burak Ozbagci.
\newblock Milnor fillable contact structures are universally tight.
\newblock {\em Math. Res. Lett.}, 17(6):1055--1063, 2010.

\bibitem{liscalensspace}
Paolo Lisca.
\newblock Lens spaces, rational balls and the ribbon conjecture.
\newblock {\em Geom. Topol.}, 11:429--472, 2007.

\bibitem{neumann}
Walter~D. Neumann.
\newblock A calculus for plumbing applied to the topology of complex surface
  singularities and degenerating complex curves.
\newblock {\em Trans. Amer. Math. Soc.}, 268(2):299--344, 1981.

\bibitem{sakuma}
Makoto Sakuma.
\newblock Surface bundles over {$S^{1}$} which are {$2$}-fold branched cyclic
  coverings of {$S^{3}$}.
\newblock {\em Math. Sem. Notes Kobe Univ.}, 9(1):159--180, 1981.

\bibitem{savkballs}
O{\u{g}}uz {\c{S}}avk.
\newblock More {B}rieskorn spheres bounding rational balls.
\newblock {\em https://arxiv.org/pdf/1912.04654.pdf}.

\bibitem{simone4}
Jonathan Simone.
\newblock Classification of torus bundles that bound rational homology circles.
\newblock {\em https://arxiv.org/pdf/2006.14986.pdf}.

\bibitem{simone1}
Jonathan Simone.
\newblock Symplectically replacing plumbings with {E}uler characteristic 2
  4-manifolds.
\newblock {\em Journal of Symplectic Geometry, to appear}, 18(5), 2020.

\end{thebibliography}

\end{document}